\documentclass[preprint,12pt,authoryear]{elsarticle}

\usepackage{amssymb}
\usepackage{amsthm}

\newcommand\A{\mathcal{A}}
\newcommand\C{\mathcal{C}}
\newcommand\D{\mathcal{D}}

\newcommand\R{\mathbb{R}}

\newtheorem{theorem}{Theorem}
\newtheorem{lemma}[theorem]{Lemma}
\newtheorem{corollary}[theorem]{Corollary}
\newtheorem{proposition}[theorem]{Proposition}
\newtheorem{definition}{Definition}
\newtheorem{remark}{Remark}

\def\ve{{\varepsilon}}
\def\eps{{\varepsilon}}

\def\s{{\mathbf{s}}}

\def\D{{\cal D}}

\def\Ga{{\cal G}_\alpha}

\def\Ger{{\cal G}_{\eps,\alpha}}
\def\J{{\cal J}}

\def\.{{\;}}
\def\bv{{{\tt{BV}}(\Omega)}}
\def\lzo{{L_2(\Omega)}}

\def\phia{{\phi_{a}}}
\def\phic{{\phi_{s}}}
\def\vb{\textbf{b}}
\def\t{\triangle}

\journal{Journal of Mathematical Analysis and Applications}

\begin{document}

\begin{frontmatter}



\title{Regularization approaches for quantitative Photoacoustic tomography using the radiative transfer equation}


\author{A.~De Cezaro and F.~Travessini De Cezaro}

\address{Institute of Mathematics Statistics and Physics,
Federal University of Rio Grande, Av. Italia km 8, 96203-900, Rio
Grande, Brazil ({\tt decezaromtm@gmail.com}).}

\begin{abstract}
Quantitative Photoacoustic tomography (QPAT) is an emerging medical
imaging modality which offers the possibility of combining the high
resolution of the acoustic waves and large contrast of optical waves by
quantifying the molecular concentration in biological tissue.

In this paper, we prove properties of the forward operator that
associate optical parameters from measurements of a reconstructed
Photoacoustic image. This is often referred to as the optical
inverse problem, that is nonlinear and ill-posed. The proved
properties of the forward operator provide sufficient conditions to
show regularized properties of approximated solutions obtained by
Tikhonov-type approaches. The proposed Tikhonov- type approaches
analyzed in this contribution are concerned with physical and
numerical issues as well as with \textit{a priori} information on
the smoothness of the optical coefficients for with (PAT) is
particularly a well-suited imaging modality.

\end{abstract}

\begin{keyword}
Quantitative Photoacoustic Tomography \sep Tikhonov-type regularization \sep
Convergence \sep Stability.
\PACS 65N12 \sep 65R32 \sep 65F32
\end{keyword}

\end{frontmatter}




\section{Introduction}\label{sec:1}

Photoacoustic tomography (PAT) is an emerging medical imaging
modality which combines the high contrast of the optical waves and the
large resolution of the acoustic waves by a laser-generated
ultrasound. From the practical point of view, a PAT image is
reconstructed from temporal surface measurements of propagated
photoacoustic waves which are generated by illuminating an optically
 absorbing and scattering medium with short
pulses of variable or near-infrared light. As the optical radiation
propagates, a fraction of its energy is absorbed by the chromophores
within the tissue and generates a small and localized heating and
pressure of the underlying medium. Due the elasticity of soft
tissue, the given perturbation of the physical conditions produces a
spatial dependent ultrasound signal that propagated through the
domain of interest. This physical phenomenon is often called the
photoacoustic effect. The resulting emitted pressure wave is
measured by ultrasonic transducers located on the surface of the
domain of interest as a function of time for which one tries to
recover the acoustic source that gives us information about the rate
of absorption at each point within the body, creating an image.

The image reconstruction in (PAT) involves the solution of two
inverse problems: the first consists in reconstructing the amount of deposit
energy from surface measurements of the propagated acoustic waves. In this issue, there are
many results on both theoretical and
numerical, e.g. \cite{ABJK2011, BR11, BRUZ11,BU10, KK2008, KS2012,
JSUZ2011, SU2009, TZC2010, TCKA12} and references therein. 
In particular, in \cite{SU2009} an if-and-only-if conditions for
uniqueness and stability is given and an explicit formula of a
convergent Neumann series type is derived. Under the assumptions of
constant sound speed and odd dimension, a time-reversal reconstruction
algorithm is presented in \cite{TZC2010}. For even dimension, we only can
expect an approximated solution \cite{JSUZ2011, SU2009}. Provided
that the first inverse problem is well-studied, we concentrate our
effort in the second inverse problem in quantitative PAT (QPAT): To
determines the chromosphere concentration distributions from the
reconstructed PAT images. Since chromosphere concentrations are
linearly related to the optical absorption concentration via the
chromophores' molar absorption coefficients, we are looking for to determine a
quantitative accurate estimate for the absorption coefficient from
the measured energy map.

Many recent contributions attempt to recover the absorption
coefficient in PAT, e.g \cite{ABJK2011, BR11, BRUZ11, BU10, KS2012}
and references therein. However, the diffusive nature of light
propagation in a turbid medium such as biological tissue means that
information is quickly lost when it travels further away from the
source \cite{KK2008, STCA13, TCKA12}. Therefore, the measured energy
maps depend of both optical coefficients, absorption and scattering.
Since neither are likely to be known or easily measured, the (QPAT)
seeks for recovering quantitative estimates of both coefficients
simultaneously.

Our paper is organized as follows: In the remained part of this section,
differently of the early model-based inversion using the
diffusion approximation \cite{ABJK2011, BR11, BRUZ11, BU10, KS2012}
for PAT modeling, we introduce the full radiative transfer equation model
to light propagation \cite{DautrayLions93, KK2008, STCA13, TCKA12}.
In Section~\ref{section2}, we collect the results of existence and
uniqueness for a solution of the radiative transfer equation as
well as regularity of such a solution which will be fundamental for
the following analysis. Details can be found in \cite{DautrayLions93}.

The novelty of the paper starts in Section~\ref{Subsec:Forward}
where we prove continuity, compactness and Fr\'echet
differentiability for the forward operator, provided that the
absorption and scattering coefficients are embedding in appropriated
topologies. These properties allow us to conclude the ill-posedness of the inverse
problem and the necessity of introducing regularization approaches
to obtain stable approximate solutions. In Section~\ref{sec:regularization} we propose Tikhonov-type
regularization approaches regards \textit{a priori} smoothness
assumptions on the coefficients. We prove standard regularization
properties \cite{EngHanNeu96} of the approximated solutions, i.e.,
we prove convergence and stability with respect the measured data.
In particular, we proposed a level set regularization approach for
the case in which the coefficients are assumed to be piecewise constant:
a particularly well-suited to imaging the blood vasculature for with
(PAT) is widely used. Although we do not show numerical results in
this contribution, we provide a glimpse of the numerical derivation
in Section~\ref{sec:numerical}, we are supporting the numerical implementation in \cite{STCA13}
 for which by the best of the authors acknowledge there was not a fully
regularization theory derived before. In Section~\ref{sec:conclusion}, we formulate some conclusions and
future works. In the remainder part of this introduction we present
the radiative transfer equation which will be our forward model for
(PAT) and we also introduce some notation.

\subsection{The forward model: radiative transfer
equation}\label{subsec:RTE}

Recently, many advances and several inversion methods have been
proposed for (QPAT) \cite{ABJK2011, BR11, BRUZ11, BU10, KS2012}.
The proposed models typically assume the diffusion
approximation to the radiative transfer equation (RTE), i.e. they
assume that the propagation of light throughout the tissue is
near-isotropic. However, lights propagation in a turbid regions is highly
anisotropic in regions close to light sources and it does not
behave diffusively until travel away to the source location. Hence,
the diffusion approximation does not provide a suitable accurate
model in a significant portion of the image which often contains
information of great interest \cite{STCA13, TCKA12}.

In this approach we consider the second inverse problem in PAT in a region of interest $\Omega \subset
\mathcal{R}^n$, with $n=2,3$. Moreover, we assume that light transport in a turbid medium may
be modeled analytically using the radiative transfer equation (RTE)
\begin{equation}\label{eq:RTE}
(\s \cdot \nabla + \mu_a(x) + \mu_s(x))u(x,\s) -
\mu_s(x)\int_{S^{n-1}} \Theta(\s, \s') u(x, \s') ds' = q(x, \s)\,.
\end{equation}

The integro-differential equation~(\ref{eq:RTE}) represents the
conservation of energy in a particular control volume. The physical
interpretation of equation~(\ref{eq:RTE}) can be read as follows: the
light travels throughout a region in a particular direction $\s \in
S^{n-1} $. The energy can be lost through the absorption and scattering
of a photon out of the direction of interest or the net outflow of
the region due to the gradient, and can be gained by the scattering
of a photon into the direction of interest or from any light sources
in the medium. The probability per unit length of an absorbing and
scattering event are represented, respectively, by the absorption
coefficient $\mu_a(x)$ and the scattering coefficient $\mu_s(x)$, at
a point $x \in \Omega$. $\Theta(\s, \s')$ is the scattering phase
function which is a probability density function that describes the
probability that a photon traveling in a direction $\s$  will be
scattered into a direction $\s'$. Also, $q$ represents the light source. Since
light propagates faster than sound, the optical propagation and
absorption can be treated as instantaneous on an acoustics
timescale, the quantity of interest is the time- integrated radiance
$u (x, \s)$, which is the energy per unity of area at a point $x \in
\Omega$ in a direction $\s \in S^ {n-1} $. Assuming that there are no
photons traveling in an inward direction at the boundary $\partial
\Omega$ except at the source position $ \Gamma_s \subset \partial
\Omega$, we can complete the (RTE) equation~(\ref{eq:RTE}) with the
given boundary condition
\begin{eqnarray}\label{eq:RTE_bc}
u(x,\s) =  \left\{ \begin{array}{ccccc}
            u_0(x,\s)\,, &  x \in  \cup \Gamma_s\, & \s \cdot \eta <
            0\\
            0\,,      & x \in  \partial \Omega - \cup \Gamma_s\, &
\s \cdot \eta < 0\,,
               \end{array}
               \right.
\end{eqnarray}
where $u_0$ is the boundary source and $\eta$ is a unitary vector
normal to $\partial \Omega$. In this paper, we will assume that
$u_0$ has a compact support in $\partial \Omega$. This property is
necessary for the existence of the trace operator  in appropriated
spaces (see Proposition~\ref{pr1} below).

The total energy at a point $x \in \Omega$ shall be equal to the
integral of the total energy per unit of area $u(x, \s)$ over all
directions, i.e.,
\begin{equation}\label{eq:fluence}
U(x) = \int_{S^{n-1}} u(x, \s) d\s\,,
\end{equation}
and it is often called the fluence.

From thermodynamic considerations we are allowed to write the
initial pressure $p_0$ arising from this optical absorption as
\begin{equation}\label{eq:p0}
p_0(x) = \Pi(x) F(x)\,,
\end{equation}
where
\begin{equation}\label{eq:h}
F(x) = F(\mu_a(x), \mu_s(x)): = \mu_a(x) U(\mu_a(x), \mu_s(x))\,,
\end{equation}
is the amount of optical energy absorbed per unit volume in
$\Omega$. $\Pi$ represents the Gr\"{u}neisen parameter,  which is a dimensionless,
tissue-specific property responsible for the photoacoustic
efficiency, i.e, representing the conversion efficiency of the heat
energy into pressure.

The first inverse problem in PAT is recovering the initial pressure
$p_0(x)$ from measurements of the acoustic pressure $p(x,t)$ over
some arbitrary measurement surface. When the sound speed $c_s$ and
the density are uniform and the optical excitation is regarded as
instantaneous, the acoustic propagation may be well described by
initial value problem for the homogeneous wave equation
\begin{equation}\label{eq:wave}
p_{tt} - c_s^2 \Delta p = 0\,,
\end{equation}
and the initial conditions are given by
\begin{equation}\label{eq:initial_wave}
p(x,0) = p_0(x)\,, \qquad p_t(x,0) = 0\,.
\end{equation}

When the sound speed $c_s$ is constant, explicit formulas for
recovering $p_0$ have been obtained for a large class of geometries
of interest, e.g. \cite{ABJK2011, BR11, BRUZ11,BU10, KK2008, KS2012,
JSUZ2011, SU2009, TZC2010, TCKA12} and references therein. When the
sound speed is not constant but it is known and non-trapping conditions
are assuming the time reversal algorithm produces accurate solutions as showed
in \cite{TZC2010, SU2009}.

In this paper, we assume that the first inverse problem is solved
and that $p_0$ is known at least approximately. Normally, precise
estimations for $\Pi$ are known from recorded experiments and
we can assume that $\Pi$ is known throughout the domain. Hence,
it is straightforward to obtain a measured observed energy map
\begin{equation}
E(x) = p_0(x)/ \Pi(x)\,.
\end{equation}
On the other hand, in practical applications it is very
unlikely that one can get the exact solution
$p_0$ using any of the  well-known reconstruction
methods,  e.g. \cite{TZC2010, SU2009} and references
therein. Indeed, many sources of noise can affect the measurements,
e.g. thermal noise in the detectors.  Moreover, the recording estimations
for $\Pi$ also give us only an approximation.
Therefore, instead of assuming exact data $E \in L^2(\Omega)$ we assume to
know a measured absorbed energy map $E^\delta \in L^2(\Omega)$ satisfying
\begin{eqnarray}\label{eq:noise}
\|E - E^\delta\|_{L^2(\Omega)} \leq \delta\,,
\end{eqnarray}
where $\delta$ is a bounded for the noise level.

Since chromophore concentration is linearly related to the optical
coefficient via the chromophores molar absorption coefficient, it
can be obtained straightforwardly from $\mu_a$ provided that all
contributing chromophore types are known. Therefore, we seek to
determine a quantitative accurate estimate of $\mu_a$ from
measurements of the absorbed energy map $E^\delta$. Resuming, this is the
second inverse problem in QPAT. However, the dependence of $U$ on
$\mu_a$ and $\mu_s$ means that $E^\delta$ (and hence $p_0$) is nonlinear
related to the absorption and scattering coefficients. Since neither
of them are likely to be known or easily measured, it means that we
need to look for  recovering quantitative estimates of both coefficients simultaneously.

\paragraph{Notation: }

Throughout this presentation, we assume that $\Omega \subset
\mathcal{R}^n$, with $n=2,3$ is a bounded domain with
$\mathcal{C}^1$ boundary $\partial \Omega$. We define the
product domain $\D : = \Omega \times S$, where $S:=S^{n-1}$ is the
sphere in $\mathcal{R}^n$. $C$ will denote a generic constant, whose
values may depend on the context.

In the product space, we have the boundary $\Gamma: = \partial
\Omega \times S$ that can be decomposed into a inflow part
$\Gamma_{-} = \{(x, s) \in \Gamma\,:\, \s \cdot \eta < 0\},$ an
outflow part $\Gamma_{+} = \{(x, s) \in \Gamma\,:\, \s \cdot \eta >
0\}$, and a remainder tangential part $\Gamma_0 = \Gamma -
(\Gamma_{-} \cup \Gamma_{+})$.

In this contribution, we will consider the parameter space belongs
to the subset
\begin{equation}
D(F): = \{ (\mu_a, \mu_s)\,:\, 0 < \underline{\mu} \leq \mu_a, \mu_s
\leq \overline{\mu}\}\,,
\end{equation}
for $\underline{\mu}, \overline{\mu}$ fixed constant values under
different topologies.

For $L^p(\Omega)$ we denote the Lebesgue space of real functions on
$\Omega$ such that $\int_\Omega |f(x)|^p dx < \infty$ if $1 \leq p <
\infty$ and $ess\sup|f(x)| \leq \infty$ for $p=\infty$. We also
denote by $W^{k,p}(\Omega)$ the Sobolev space of all functions whose
all the derivatives up to the order $k$ belongs to $L^p(\Omega)$. In
particular, for $p=2$ we have the Hilbert spaces $W^{k,2}(\Omega) =
H^k(\Omega)$. Moreover, $C^\infty_0(X)$ denote the set of infinity
continuous differentiable functions which compact support in $X$.

To avoid possible confusions, we shall introduce also the Banach
space $L^p(\D)$ ($1\leq p < \infty$) defined on the space of
Lebesgue function for the product measure $dx d\s$ such that
$\|f\|^p_{L^p(\D)} = \int_\Omega \int_S |f(x,\s)|^p dx d\s <
\infty$. Moreover, $W^p(\D): = \{ f \in L^p(\D) : \s \cdot \nabla f
\in L^p(\D)\}$ denotes the Banach space where the
integro-differential operator in equation~(\ref{eq:RTE}) will be
well-posed.

As we will see, because of physical reasons the natural spaces for the
radiance and for the fluence are $L^1(\D)$ and $L^1(\Omega)$,
respectively. Indeed, we can define the so-called transport operator
$T$ as
\begin{equation}\label{eq:T}
T u(x, \s) = (\s \cdot \nabla + \mu_a(x) + \mu_s(x)) u(x, \s) -
\mu_s(x) \int_S \Theta(\s, \s') u(x, \s') d\s'\,,
\end{equation}
which it is naturally defined in the space $L^1(D)$ of integrable
functions and its domain $D(T)$ is given by
\begin{eqnarray*}
D(T): = \left\{  u \in L^1(D)\, :\, Tu \in L^1(\D) \mbox{ and }
u(x,s) = 0, \mbox{ a.e. } (x,\s) \in \Gamma_{-} \right\} \,.
\end{eqnarray*}
Therefore, since the absorption and scattering coefficient belong to
$D(F)$, then is easy to see that $D(T) \subset W^p(D)$. Of course, the
trace operator must make sense in such topology. It will be
guarantee in Lemma~(\ref{lemma1-DL-v6}) below.

However, for numerical as well as theoretical reasons, other
$L^p$-spaces play an important rule in the game. In particular, the
development of computational schemes in a Hilbert space makes
$L^2(\D)$ with the inner product
\begin{eqnarray*}
( u, v)_{L^2(\D)}: = \int_\Omega \int_S u(x,\s) v(x,\s) dx d\s
\end{eqnarray*}
a very suitable candidate.

We will denote the product $X \times X$ of the two Banach
spaces by $[X]^2$.

\section{On the existence and regularity of a solution of RTE
equation} \label{section2}

In this section we revisit some well-known results of existence
and regularity for the (RTE)
equation~(\ref{eq:RTE})-(\ref{eq:RTE_bc}), for which we suggest the
reference \cite{DautrayLions93}. 

The first result in this direction is concerned with the trace
operator and the well posedness of the boundary
condition~(\ref{eq:RTE_bc}).

\begin{lemma}
The inflow and the outflow boundaries $\Gamma_{-}$ and $\Gamma_{+}$
are open subsets of $\Gamma$ and $\Gamma_0$ is a closed subset of
$\Gamma$ with $ (2n - 2) $-dimensional zero measure.
\end{lemma}
\begin{proof}
The $\mathcal{C}^1$ regularity of $\partial \Omega$ implies that:
(i) map $(x,\s) \longmapsto \s \cdot \eta$ is continuous and (ii)
$\partial \Omega$ is locally diffeomorphic to a subset of $R^{n-1}$.
From (i) we have that $\Gamma_0$ is closed and $\Gamma_{-}$ and
$\Gamma_{+}$ are open subsets of $\Gamma$. From (ii) and the
standard product structure of $\Gamma_0, \Gamma_{-}, \Gamma_{+}$ we
have the assertions.
\end{proof}

Previews lemma allows to identify measurable functions on $\Gamma$
with functions defined in $\Gamma_{-} \cup \Gamma_{+}$. However,
if $u \in W^p(\D)$, it is not true that the trace
$u|_{\Gamma_{-}}$ (respectively $u|_{\Gamma_{+}}$) satisfies
$$ \int_{\Gamma_{-}} \s \cdot \eta |u|^p dx d\s < \infty$$
even for $p= 2$, see \cite{DautrayLions93}. But the result is true
if $u$ has a compact support in $\Gamma_{-}$ as shown by the
following proposition.
\begin{proposition}\label{pr1}
Let $\mathcal{K}$ be a compact subset of $\Gamma_{-}$ (resp.
$\Gamma_{+}$). Then the trace map $ u \longmapsto u|_{\mathcal{K}}$
defined in $C^\infty_0(\overline{\D})$ is extended by continuity to
a bounded linear operator from $W^p(\D)$ to $L^p(\mathcal{K})$.
\end{proposition}
\begin{proof}
The proof is given by \cite[Theorem~1, pp 220]{DautrayLions93}.
\end{proof}

\begin{remark}
The well-definition of the boundary condition~(\ref{eq:RTE_bc})
follows from the assumption that the support of $u_0$ is compact
embedding in $\Gamma_{-}$ and Proposition~\ref{pr1}.
\end{remark}

We will prove the next lemma in details, since similar techniques
will be used later. A similar result is given in \cite[Lemma~1, pp.
227]{DautrayLions93}

\begin{lemma}\label{lemma1-DL-v6}
The scattering operator
\begin{equation}
K u : = \mu_a \int_S \Theta (x, \s, \s') u(x,\s) d\s'
\end{equation}
is linear and continuous from $L^p(\D)$ in itself, for $ 1\leq p \leq \infty$.
\end{lemma}
\begin{proof}
The linearity of $K$ follows immediately.

Since $\Theta$ is a probability kernel, it follows that $\Theta \geq
0$ and $\int_S \Theta(x,\s, \s') \leq 1$. Hence, given the uniformly
bounded of the coefficient, the assertion for $p = 1$ and $p =
\infty$ follows. Let´s consider the other cases. Using the
H\"{o}lder inequality ($1/p + 1/p' = 1$) we have
\begin{eqnarray*}
\|K u\|^p_{L^p(\D)} & = & \int_D \left| \mu_a \int_S \Theta(x,\s,
\s') u(x,\s') d\s'\right|^p dx d\s \\
& & \leq \overline{\mu}^p \int_D \left( \int_S \Theta(x,\s, \s')^{1
- 1/p} \Theta(x,\s, \s')^p
|u(x,\s')| d\s'\right)^p dx d\s \\
& & \leq \overline{\mu}^p \int_D \left( \int_S \Theta(x,\s, \s') d\s
\right)^{p/p'} \left( \int_S \Theta(x,\s, \s') |u(x,\s')|^p
d\s'\right) dx d\s \\
& & \leq \overline{\mu}^p \int_D \int_S \Theta(x,\s, \s')d\s\,
|u(x,\s')|^p  dx d\s' \leq \overline{\mu}^p \int_D |u(x,\s')|^p dx
d\s'\,.
\end{eqnarray*}
\end{proof}

Next we will present the regularity of solutions of the (RTE)
equation~(\ref{eq:RTE}).
\begin{theorem}\label{theo:4-DL,p241}
Let $q \in L^p(\D)$ and the coefficient $(\mu_a, \mu_s) \in D(F)$,
such that  $u_0 \in L^p(\Gamma_{-})$ (meaning that $u_0$ belong to
the spaces of a trace of a function $v \in W^p(\D)$ as in
Proposition~\ref{pr1}) . Then, there exists a unique $u \in L^p(\D)$
solution of (\ref{eq:RTE})-(\ref{eq:RTE_bc}), for $p \in [1,
\infty]$.

Moreover, we have the bound
$$
\|u\|_{L^p(\D)} \leq C \left( \|q\|_{L^p(\D)} +
\|u_0\|_{L^p(\Gamma_{-})} \right)\,,
$$
with $C$ depending only of the boundedness of the coefficients and
on $\D$.
\end{theorem}
\begin{proof}
Since $(\mu_a, \mu_s) \in D(F)$ and $\int_S \Theta(x, \s, \s') \leq
1$, it follows that
\begin{eqnarray*}
\mu_a + \mu_s - \mu_a \int_S \Theta(x, \s, \s') & & \geq \mu_s \geq
\underline{\mu} > 0\, \qquad \mbox{and} \\
\mu_a \int_S \Theta(x, \s, \s') & & \leq \mu_a \leq \beta (\mu_a +
\mu_s) \,\,\mbox{ for some }\, 0\leq \beta < 1\,.
\end{eqnarray*}
Therefore, the assumptions on Theorem~4, Proposition~5 and
Proposition~6 in \cite[Chapter XXI]{DautrayLions93} are satisfied.
They guarantee, respectively, the existence of a unique solution $u
\in L^p(\D)$, with $1 < p < \infty$, $p=1$ and $p =\infty$ for
equation~(\ref{eq:RTE}) with absorbing boundary condition $(u_0=0)$,
satisfying $\|u\|_{L^p(\D)} \leq C \|q\|_{L^p(\D)}$\,.

Now, by using the lifting of the boundary condition $u_0$, the
linearity of (\ref{eq:RTE}) and the superposition principle, the
existences of a unique solution $u \in L^p(\Omega)$ for the
non-homogeneous boundary condition
equation~(\ref{eq:RTE})-(\ref{eq:RTE_bc}) follows from the absorbing
boundary condition results.
\end{proof}

Although the natural space of definition of the transport operator
$T$ in (\ref{eq:T}) is $L^1(\D)$, Theorem~\ref{theo:4-DL,p241} and
Lemma~\ref{lemma1-DL-v6} show that the transport operator is also
well defined in $W^p(\D)$, as long as $q, u_0$ belong to a enough
regular space. Moreover, we can see that the operator $T$ has an
adjoint in $L^{p'}(\D)$ given by
\begin{equation}\label{eq:T-adjoint}
T^* v = (- \s \cdot \nabla + \mu_a + \mu_s) v - \mu_s \int_S
\Theta(\s') v(\s') d\s'\,,
\end{equation}
such that the integro-differential equation
\begin{equation}
T^* v = \tilde{q}\,, \qquad   v|_{\Gamma_{+}} = g\
\end{equation}
has a unique solution in $L^{p'}(\D)$, for $1/p+ 1/p'= 1$, for any
$\tilde{q} \in L^{p'}(\Omega)$ and $g \in L^{p'}(\Gamma_{+})$ with
compact support. A detailed proof can be found in \cite[Section~
3.3]{DautrayLions93}.

\section{Properties of the forward operator}\label{Subsec:Forward}

In PAT, the nonlinear operator equation $D(F) \ni (\mu_a, \mu_s)
\longmapsto F(\mu_a, \mu_s)$ naturally maps the absorbed energy
given by equation~(\ref{eq:h}) into $L^1(\Omega)$. However, as we
showed in Theorem~\ref{theo:4-DL,p241}, the radiance $u$ may belong to
$L^p(\D)$ if source $q$ and the boundary condition $u_0$ are smooth
enough. It makes sense (from the numerical as well as from
the theoretical point of view) to looking for the operator equation
\begin{eqnarray}\label{eq:operator}
F\,:\, & \D(F) \longrightarrow L^p(\Omega)\qquad 1 \leq p \leq
\infty\,\nonumber\\
      &(\mu_a, \mu_s)
\longmapsto F(\mu_a, \mu_s)
\end{eqnarray}

In this section, we show the properties of the operator
equation~(\ref{eq:operator}) by considering $D (F) $ in different
topologies. Among those, we are interested in proving continuity,
compactness and Fr\'echet differentiability which allows to prove
convergence and stability of the different regularization approaches in
Section~\ref{sec:regularization}.

Let $(\mu_a, \mu_s)$ and $ (\tilde{\mu}_a, \tilde{\mu}_s) \in D(F)$ and $u
= u(\mu_a, \mu_s)$ and $v = u(\tilde{\mu}_a, \tilde{\mu}_s)$ the
respectively unique solutions of (\ref{eq:RTE})-(\ref{eq:RTE_bc}),
with source $q \in L^p(\D)$ and boundary condition $u_0 \in
L^p(\Gamma_{-})$. By linearity of (\ref{eq:RTE})-(\ref{eq:RTE_bc}),
we have that $ w = u - v$ satisfies
\begin{equation}\label{eq:RTE-w}
%
T w = ( (\mu_a - \tilde{\mu}_a) + (\mu_s - \tilde{\mu}_s)) v + (
\mu_s -\tilde{\mu}_s) \int_S \Theta(\s, \s') v(\s') d\s'\,.
\end{equation}
with absorbing boundary condition. Notice that, from
Theorem~\ref{theo:4-DL,p241} and Lemma~\ref{lemma1-DL-v6} there
exists a unique solution $w \in L^p(\Omega)$ for the
integro-differential equation~(\ref{eq:RTE-w}).

Let us consider $p' =2$ for a while. Then, from the proof of
Theorem~4, pp. 241 in \cite{DautrayLions93}, we have that
\begin{eqnarray}\label{eq:coercivityL2}
\underline{\mu} \|w\|^2_{L^2(\D)} & & \leq \langle (\mu_a + \mu_s)
w - \mu_s \int_S \Theta(\s, \s') w(\s') d\s' , w \rangle_{L_2(\D)}\\
& & \leq \left\langle (\s \cdot \nabla + \mu_a + \mu_s) w - \mu_s
\int_S \Theta(\s, \s') w(\s') d\s' , w \right\rangle_{L_2(\D)}\,.
\nonumber
\end{eqnarray}

By multiplying equation~(\ref{eq:RTE-w}) by $w$, integrate over $\D$
in both sides and using (\ref{eq:coercivityL2}) we get that
\begin{eqnarray}\label{eq:eq}
\underline{\mu} \|w\|^2_{L^2(\D)} \leq & & \int_\D ( |\mu_a -
\tilde{\mu}_a| + |\mu_s - \tilde{\mu}_s|) |v|| w| dx d\s \\ & & +
\int_\D \left( | \mu_s -\tilde{\mu}_s| \int_S \Theta(\s, \s')
|v(\s')| d\s' \right) |w| dx d\s \,. \nonumber
\end{eqnarray}
Using the H\"{o}lder inequality for $1/p +1/r + 1/p' = 1$ and the
techniques on Lemma~\ref{lemma1-DL-v6}, it follows from
(\ref{eq:eq}) that
\begin{equation}\label{eq:eq1}
\underline{\mu} \|w\|^2_{L^2(\D)} \leq C \left( \|\mu_a -
\tilde{\mu}_a\|_{L^r(\Omega)} + \|\mu_s -
\tilde{\mu}_s\|_{L^r(\Omega)}\right) \|v\|^p_{L^p(\D)}
\|w\|_{L^2(\D)} \,.
\end{equation}

Now, from Theorem~\ref{theo:4-DL,p241} we conclude that
\begin{equation}\label{eq:eq2}
\|w\|_{L^2(\D)} \leq C (\|q\|, \|u_0\|) \left( \|\mu_a -
\tilde{\mu}_a\|_{L^r(\Omega)} + \|\mu_s -
\tilde{\mu}_s\|_{L^r(\Omega)}\right)\,.
\end{equation}

\begin{remark}\label{remark:coercivity-p'}
The coercivity of the bilinear form $\langle  Tw , w
\rangle_{L^{p'}(\D)}$ follows $ - \langle (\mu_a +
\mu_s) w - \mu_s \int_S \Theta(\s, \s') w(\s') d\s' , w
\rangle_{L^{p'}(\D)} \leq - \underline{\mu}\|w\|^{p'}_{L^{p'}(\D)}$
for $p'\in ]1, \infty[$. The proof is analogous to Theorem 4, pp.
241 \cite{DautrayLions93}, with the help of the H\"{o}lder´s
inequality replacing the Cauchy-Schwarz inequality in the case
$p'=2$.
\end{remark}

From Remark~\ref{remark:coercivity-p'} and using the same arguments
in equations~(\ref{eq:eq})-(\ref{eq:eq2}) for $1/p+ 1/r+ 1/p' = 1$,
we deduce that
\begin{equation}\label{eq:eq3}
\|w\|_{L^{p'}(\D)} \leq C (\|q\|_{L^p(\D)}, \|u_0\|_{L^p(\D)})
\left( \|\mu_a - \tilde{\mu}_a\|_{L^r(\Omega)} + \|\mu_s -
\tilde{\mu}_s\|_{L^r(\Omega)}\right)\,.
\end{equation}

\begin{remark}\label{remark:regularity}
Note that Equation~(\ref{eq:eq3}) reflects the amount of regularity
that we should expect on the coefficients and on the respective
solution of (\ref{eq:RTE}) in order to get continuity of the
forward operator in a particular space. In particular, from
Theorem~\ref{theo:4-DL,p241}, if the source and boundary conditions
are in $L^p(\D)$ and $ p \to \infty$, we can expect continuity
of the forward operator in $L^r(\Omega)$ for $r \to 1$.
\end{remark}

Indeed, we have the following Theorem.

\begin{theorem}\label{theo:continuity}
Let $p' \in ]1, \infty[$. Then the forward operator $F$ defined in
(\ref{eq:operator}) is continuous from $D(F)$ to $L^{p'}(\Omega)$,
with $D(F)$ consider in the $[L^r(\Omega)]^2$-topology, for $1/p +
1/p'+1/r = 1$.
\end{theorem}
\begin{proof}
As before, lets $(\mu_a, \mu_s), (\tilde{\mu}_a, \tilde{\mu}_s) \in
D(F)$ and $u, \tilde{u} \in L^p(\D)$ the respective solution of
(\ref{eq:RTE})-(\ref{eq:RTE_bc}) (from Theorem~\ref{theo:4-DL,p241},
the $L^p$ regularity of $u, \tilde{u}$ is reflected into the
regularity of the source and the boundary condition).

Notice that
\begin{eqnarray*}
F(\mu_a, \mu_s) - F(\tilde{\mu}_a, \tilde{\mu}_s) & & = \mu_a \int_S
u(\cdot; \s) d\s - \tilde{\mu}_a \int_S \tilde{u}(\cdot; \s) d\s \\
& & = (\mu_a - \tilde{\mu}_a)\int_S u(\cdot; \s) d\s - \tilde{\mu}_a
\int_S (\tilde{u}(\cdot; \s)  - u(\cdot, \s)) d\s\,. \nonumber
\end{eqnarray*}
Therefore, using the same arguments in the proof of
Lemma~\ref{lemma1-DL-v6}, we have that
\begin{eqnarray}\label{eq:4}
\int_\Omega & & \left| F(\mu_a, \mu_s)  - F(\tilde{\mu}_a,
\tilde{\mu}_s) \right|^{p'} dx \\ & & \leq \int_\Omega  \left(|\mu_a
- \tilde{\mu}_a|\int_S |u(\cdot; \s)| d\s + \overline{\mu} \int_S
|\tilde{u}(\cdot; \s)  - u(\cdot, \s)| d\s \right)^{p'}dx  \nonumber \\
& & \leq C \left( \|\mu_a - \tilde{\mu}_a\|^{r/p'}_{L^r(\Omega)}
\|u\|^{p'}_{L^{p'}(\D)} + \|\tilde{u} -
u\|^{p'}_{L^{p'}(\D)}\right)\,.\nonumber
\end{eqnarray}
Now, Theorem~\ref{theo:4-DL,p241} and the inequality~(\ref{eq:eq3})
conclude the assertion.
\end{proof}

\begin{remark}\label{remark:continuity}
There are some cases that we would like to point out in
Theorem~\ref{theo:continuity}.

\noindent \textbf{Case $p'=1$:} As we commented before $L^1(\Omega)$
is the natural topology for the fluence $U$ and so it is the natural
topology for the range of the operator $F$.

It is easy to see from equation~(\ref{eq:4}) that the
inequality
\begin{eqnarray}\label{eq:5}
\quad \int_\Omega \left| F(\mu_a, \mu_s) - F(\tilde{\mu}_a,
\tilde{\mu}_s) \right| dx  & \leq \|\mu_a -
\tilde{\mu}_a\|_{L^1(\Omega)} \|u\|_{L^\infty(\D)} + \overline{\mu}
\|\tilde{u} - u\|_{L^{1}(\D)}\,,
\end{eqnarray}
holds true, if the respective solution of equation~(\ref{eq:RTE}) -
(\ref{eq:RTE_bc}) is in $L^\infty(\D)$. Remember that, from
Theorem~\ref{theo:4-DL,p241}, a sufficient condition for the
$L^\infty$ regularity of a solution of (\ref{eq:RTE}) -
(\ref{eq:RTE_bc}) is that $q, u_0 \in L^\infty(\D)$. Since $\D$ is
bounded, we have $L^{s}(\D)$ is continuously embedding in $L^1(\D)$
and  $\|\cdot\|_{L^{1}(\D)} \leq C \|\cdot\|_{L^{s}(\D)}$, for any $
s \geq 1$. Using this fact in (\ref{eq:5}), for $s=p'$ w.r.t. the norm of
$u - \tilde{u}$ and $s = r$ w.r.t. the norm of the coefficient and
(\ref{eq:eq3}), we have that
\begin{eqnarray}
\int_\Omega \left| F(\mu_a, \mu_s) - F(\tilde{\mu}_a, \tilde{\mu}_s)
\right| dx  & \leq C \left( \|\mu_a - \tilde{\mu}_a\|_{L^r(\Omega)}
+ \|\mu_s - \tilde{\mu}_s\|_{L^r(\Omega)} \right)\,. \label{eq:6}
\end{eqnarray}

\noindent \textbf{Case $p'=2$:} For the numerical point of view is
important to have an inner product to help in the computational
implementation \cite{STCA13}.

A closer look at the proof of Theorem~\ref{theo:continuity} implies in
the last inequality that
\begin{eqnarray*}
\int_\Omega \left| F(\mu_a, \mu_s) - F(\tilde{\mu}_a, \tilde{\mu}_s)
\right|^2  dx  & \leq C \left (\|\mu_a -
\tilde{\mu}_a\|_{L^2(\Omega)} \|u\|_{L^2(\D)} + \|\tilde{u} -
u\|_{L^{2}(\D)} \right)\,,
\end{eqnarray*}
Now, we can use (\ref{eq:eq2}) to obtain that
\begin{eqnarray*}
\int_\Omega  | F(\mu_a, \mu_s)  - F(\tilde{\mu}_a, \tilde{\mu}_s)
|^2  dx  \leq C \left (\|\mu_a - \tilde{\mu}_a\|_{L^2(\Omega)} 
+ \|\mu_s - \tilde{\mu}_s\|_{L^2(\Omega)} \right)\,.
\end{eqnarray*}
Hence, for as long as we take $p = +\infty$ in the deduction of
(\ref{eq:eq2}) (this means that the respective solutions of
(\ref{eq:RTE}) - (\ref{eq:RTE_bc}) are in $L^\infty(\D)$), we have
the continuity of the operator $F$ in $L^2(\Omega)$ into itself.
\end{remark}

The next result shows the continuity of the operator $F$ for
piecewise constant coefficients.

\begin{corollary}\label{coro:continuity}
Assume that the admissible coefficient is a subset of $D(F)$ with piecewise
constant functions in $\Omega$ and $q, u_0 \in L^\infty(\D)$. Then the forward
operator $F\,:D(F) \to L^2(\Omega)$ defined in (\ref{eq:operator})
is continuous in $[L^1(\Omega)]^2$.
\end{corollary}
\begin{proof}
Notice that, for the assumption on the source and boundary
conditions of equation~(\ref{eq:RTE})-(\ref{eq:RTE_bc}), we can take
$p = \infty$. Taking $p' > 2$ (and $p = +\infty$) and following the
same arguments in Remark~\ref{remark:continuity} we easily obtain that
\begin{eqnarray}\label{eq:8}
\int_\Omega \left| F(\mu_a, \mu_s) - F(\tilde{\mu}_a, \tilde{\mu}_s)
\right|^2  dx  & \leq C \left (\|\mu_a -
\tilde{\mu}_a\|_{L^r(\Omega)} + \|\mu_s -
\tilde{\mu}_s\|_{L^r(\Omega)} \right)\,,
\end{eqnarray}
for $r = p'/(p'-1) > 1$. Therefore, there exists a $s> 0$ such that
$ r = 1+ s$.

Without lost of generality, we assume that the coefficient has only
two distinct values, let say $\mu_a(x), \mu_s(x) \in \{c_1, c_2\}\,,
x\in \Omega$. Hence,
\begin{eqnarray*}
\int_\Omega |\mu_a - \tilde{\mu}_a|^r dx = \int_\Omega |\mu_a -
\tilde{\mu}_a| |\mu_a - \tilde{\mu}_a|^s dx \leq 2 \max\{c_1,
c_2\}^s \int_\Omega |\mu_a - \tilde{\mu}_a| dx\,.
\end{eqnarray*}
The same inequality is true for the scattering coefficient.
Therefore, the assertion follows.
\end{proof}

In the following we will prove that the inverse problem is ill-posed
in appropriated topologies.

\begin{theorem}\label{theo:compactness}
Assume that the solution of (\ref{eq:RTE}) - (\ref{eq:RTE_bc}) is
in $L^p(\D)$ (see Theorem~\ref{theo:4-DL,p241} for such conditions),
$r \in ]1, \infty[$ if $n=2$ or $ r < 6$ if $n=3$ and $1/p +
1/p'+1/r = 1$. Moreover, let the operator $F\, : \, D(F) \to
L^{p'}(\Omega)$ as defined in (\ref{eq:operator}), with $D(F)$
equipped with the $[H^1(\Omega)]^2$-norm. Then $F$ is completely
continuous and weak sequentially closed in $L^{p'}(\Omega)$.
\end{theorem}
\begin{proof}
Let $\{(\mu^k_a, \mu^k_s)\}$ be a sequence in $D(F)$ weakly
convergent to $(\mu_a, \mu_s)$. Since $D(F)$ is convex and closed,
it is weakly closed. Hence the weak limit $(\mu_a, \mu_s) \in D(F)$.
Since $H^1(\Omega)$ is compact embedding in $L^r(\Omega)$ for $r$ as
in the assumption \cite{Ada75}, there exist a subsequence (that we
denote with the same index) that strongly converges in $L^r(\Omega)$.
From Theorem~\ref{theo:continuity}, we have $F(\mu^k_a, \mu^k_s) \to
F(\mu_a, \mu_s)$ in $L^{p'}(\Omega)$.
\end{proof}

\begin{remark}
The assertions of Corollary~(\ref{coro:continuity}) remained true
for $D(F)$ embedding in any space that is compact embedding in
$[L^r(\Omega)]^2$. In particular, since $\bv$ is compact embedding
in $L^r(\Omega)$ for $1 \leq r \leq 3/2$ (see \cite{EG92}), we can
consider $D(F)$ with the $[\bv]^2$-norm.
\end{remark}

We will see that the presented results on continuity and compactness
allowed us to prove regularizing properties of approximate solutions
for the inverse problem in Section~\ref{sec:regularization}. Let us
move forward and prove the differentiability of the forward
operator in suitable topologies. Differentiability is a key property
for the convergence of the iterative algorithm employed to obtain
the approximated solution of the nonlinear operator
equation~(\ref{eq:operator}).

\begin{theorem}\label{theo:dir-diff}
Let $(\mu_a, \mu_s) \in D(F)$ and $(\t \mu_a, \t \mu_s) \in
[H^1(\Omega)]^2$ such that $( \mu_a + t \t \mu_a, \mu_s + t \t
\mu_s) \in D(F)$ for $t \in \mathcal{R}$ with $|t|$ sufficiently
small. Then the directional derivative of $F$ in the direction $(\t
\mu_a, \t \mu_s)$ is given by
\begin{eqnarray}\label{eq:dir-derivative}
F'(\mu_a, \mu_s)[ \t \mu_a, \t \mu_s] = \t \mu_a U(\mu_a, \mu_s) +
\mu_a \int_S u'(\mu_a, \mu_s; \s)[\t \mu_a, \t \mu_s] d\s
\end{eqnarray}
where $u'(\mu_a, \mu_s; \s)$ satisfies the integro-differential
equation
\begin{eqnarray}\label{eq:u'}
T u' = - [ \t \mu_a + \t \mu_s ] u + \t \mu_s \int_S \Theta(\s, \s')
u(\s') d\s'\,,
\end{eqnarray}
with absorbing boundary conditions, and $u$ denotes the unique
solution of (\ref{eq:RTE})-(\ref{eq:RTE_bc}).
\end{theorem}
\begin{proof}
By linearity of equation~(\ref{eq:RTE}), it follows that the
directional derivative $u'(\mu_a, \mu_s; \s)[\t \mu_a, \t \mu_s]: =
\lim_{t \to 0} \frac{1}{t} (u( \mu_a + t \t \mu_a, \mu_s + t \t
\mu_s) - u(\mu_a, \mu_s))$ satisfies (\ref{eq:u'}).

Now, the linearity and continuity of the multiplication for $\mu_a$
in the definition of $F$ imply the assertion.
\end{proof}

\begin{lemma}
The directional derivative $F'(\mu_a, \mu_s)$ defined in
(\ref{eq:dir-derivative}) satisfies the uniform estimate
\begin{eqnarray}\label{eq:unif-bound}
\|F'(\mu_a, \mu_s) & & [ \t \mu_a, \t \mu_s]\|_{L^2(\Omega)} \\
& & \leq C (\|\t \mu_a\|_{H^1(\Omega)} + \|\t
\mu_s\|_{H^1(\Omega)})(\|q\|_{L^2(\D)} + \|u_0\|_{L^2(\D)}) \,,
\nonumber
\end{eqnarray}
where the constant $C$ depends only on $\D$ and the bounds of the
coefficients.
\end{lemma}
\begin{proof}
The result follows similarly to Theorem~\ref{theo:continuity} and
Remark~\ref{remark:continuity}.
\end{proof}

Since $D(F)$ has no interior point in the
$[H^1(\Omega)]^2$-topology, the directional derivative is not
Gateaux differentiable. However, we will prove that it defines a
linear operator that can be extended continuously to
$[H^1(\Omega)]^2$.

\begin{theorem}
Under the assumptions of Theorem~\ref{theo:dir-diff},  $
F'(\mu_a, \mu_s)[ \t \mu_a, \t \mu_s]$ has a linear and bounded
extension to $[H^1(\Omega)]^2$.
\end{theorem}
\begin{proof}
Consider the ball $B_\rho(\mu_a, \mu_s): = \{(\tilde{\mu}_a,
\tilde{\mu}_s)\,:\, \|\mu_a - \tilde{\mu}_a\|^2_{H^1(\Omega)} +
\|\mu_s - \tilde{\mu}_s\|^2_{H^1(\Omega)} \leq \rho \}$. It is easy
to see that the set $B_\rho(\mu_a, \mu_s) \cap D(F)$ is dense in
$B_\rho(\mu_a, \mu_s)$ with the $H^1$-topology. Hence, $F'(\mu_a,
\mu_s)$ is densely defined by the directional derivatives
satisfying the uniform bound (\ref{eq:unif-bound}). The uniform
boundedness principle \cite{Yosida95} implies the existence of a
unique continuous extension to $[H^1(\Omega)]^2$, which we will
denote again by $F'(\mu_a, \mu_s)$.
\end{proof}

As observed before, $D(F)$ has no interior points when equipped with
the $[H^1(\Omega)]^2$-norm. Because of that, $F$ is not necessarily
differentiable in every direction $(\t \mu_a, \t \mu_s) \in
[H^1(\Omega)]^2$. In other words, $F$ is not Gateaux differentiable.
This will not affect the convergence analysis that follows. In fact,
for such analysis we only need that the operator $F$ attains a
one-sided directional derivative at $(\mu_a, \mu_s)$ in the
directions $(\t \mu_a, \t \mu_s)$, for all $(\t \mu_a, \t \mu_s) \in
D(F)$. The sufficient condition for this to happen is $D(F)$ to be
star-like with respect to $(\mu_a, \mu_s)$. That is, for every
$(\mu_a, \mu_s) \in D(F)$ there exists $t_0 >0$ such that $(\mu_a,
\mu_s)  + t ((\t \mu_a, \t \mu_s) - (\mu_a, \mu_s)) = t(\t \mu_a, \t
\mu_s) + (1- t)(\mu_a, \mu_s) \in D(F)$ for $ 0 \leq t \leq t_0$.
Since $D(F)$ has been convex, the requirement above follows.
Moreover, the bounded linear operator $F'(\mu_a, \mu_s)$ has
properties that mimic the Gateaux derivative.

\section{Regularization approaches}\label{sec:regularization}

We are assuming the first inverse problem in PAT is solved and that
the a measured absorbed energy map $E^\delta \in L^2(\Omega)$ satisfies equation~\ref{eq:noise}.


Hence, the second inverse problem in PAT can be rewritten as
follows: find $(\mu_a, \mu_s) \in D(F)$ which correspond to the
measurements $E^\delta$. Mathematically, it means to solve the
nonlinear operator equation
\begin{eqnarray}\label{eq:op}
F(\mu_a, \mu_s)  = E^\delta\,, \qquad \mbox{s.t. } (\mu_a,
\mu_\delta) \in D(F) \mbox{ and } E^\delta \mbox{ satisfying }
(\ref{eq:noise})\,.
\end{eqnarray}

From physical reasons it is natural to assume that there exists
$(\mu^*_a, \mu^*_s) \in D (F) $ such that $F(\mu^*_a, \mu^*_s) = E$.
It means that the inverse problem has a solution. We remember that
the forward operator $F$ is compact (see
Theorem~\ref{theo:compactness}). Then it is ill-posed and some
regularization method has to be used to guarantee the existence of
stable approximated solutions. In this contribution we consider
Tikhonov-type regularization strategies for obtaining a stable
approximated solution for the second (QPAT) inverse problem.

Since PAT is particularly used for imaging different tissue regions,
it is common to see different requirements for the smoothness of the
structures. This information is crucial for proposing appropriated
the regularization term in the Tikhonov-type approaches that
reflects the expected smoothness of the coefficients. In the
following, we will use the smoothness of the coefficients has the
\textit{a priori} information in the Tikhonov approach.

\subsection{Tikhonov-type regularization: smooth
coefficients}\label{subsec:smooth_coeff}

In the following, we consider the standard Tikhonov regularization,
i.e., we define an approximated solution $(\mu^{\alpha, \delta}_a,
\mu^{\alpha, \delta}_s)$ as a minimizer of the Tikhonov functional
\begin{eqnarray}\label{eq:Tikho-func-smooth}
 \J_{\alpha}(\mu_a, \mu_s): & & = \frac{1}{p}\|F(\mu_a, \mu_s) -
E^\delta\|^p_{L^p(\Omega)} \\
\quad &&  + \alpha \left(\|\mu_a - \mu_{a,0}\|^2_{H^1(\Omega)} +
\|\mu_s - \mu_{s,0}\|^2_{H^1(\Omega)}\right) \nonumber
\end{eqnarray}
subject to $ (\mu_a, \mu_s) \in D(F) \cap H^1(\Omega)$ and $1 \leq p
\leq 2$. The element $(\mu_{a,0}, \mu_{s,0} \in [H^1(\Omega)]^2$
serves as an \textit{a-priori} guess for the unknown parameters and
$\alpha
> 0$ is the regularization parameter.

\begin{remark}\label{remark:Tikho1}
The restriction on $p \in [1, 2]$ reflects the following estimate:
Since $\Omega$ is bounded, $L^2(\Omega)$ is continuous embedding in
$L^p(\Omega)$, for $p \in [1, 2]$ and $\|\cdot\|_{L^p(\Omega)} \leq
C \|\cdot\|_{L^2(\Omega)}$.  Assume that $(\mu_a^\dag, \mu_s^\dag)
\in D(F) \cap H^1(\Omega)$ is a solution of (\ref{eq:operator}).
Then, for any minimizer $(\mu^{\alpha, \delta}_a, \mu^{\alpha,
\delta}_s)$ of $\J_\alpha$ we have that
\begin{eqnarray*}
\J_\alpha(\mu^{\alpha, \delta}_a, \mu^{\alpha, \delta}_s) &  \leq
\J_\alpha(\mu^\dag_a, \mu^\dag_s) 
\leq \frac{1}{p} \delta^p + \alpha\left(\|\mu^\dag_a -
\mu_{a,0}\|^2_{H^1(\Omega)} + \|\mu^\dag_s -
\mu_{s,0}\|^2_{H^1(\Omega)}\right)\,.
\end{eqnarray*}
This estimate implies that if the perturbation in the measurements
goes to zero and the regularization parameter $\alpha$ is chosen
appropriately, then the regularized solutions can be shown to
converge to a solution of the inverse problem.
\end{remark}

The estimate in Remark~\ref{remark:Tikho1}, the continuity and
compactness of the forward operator $F$ in
Theorem~\ref{theo:continuity} and Theorem~\ref{theo:compactness},
and some minor modifications from the standard Tikhonov
regularization theory for nonlinear inverse problems are all that we
need  to show stability and convergence of the approximated
solutions. For details of the proofs see
\cite[Chapter~10]{EngHanNeu96}.
The next result states, existence, stability and convergence of an
approximated solutions of the inverse problem w.r.t. the noise in
the data.
\begin{theorem} Let the Tikhonov functional $\J_\alpha$ defined in
(\ref{eq:Tikho-func-smooth}), $p,q,r$ chosen as in
Theorem~\ref{theo:continuity} and Theorem~\ref{theo:compactness},
then:

\noindent \textbf{[Existence of a minimizer]} For any $\alpha > 0$,
the Tikhonov functional $\J_\alpha$ has a minimizer in $D(F) \cap
[H^1(\Omega)]^2$.

\noindent \textbf{[Stability]} Let be $\alpha > 0$ and let be
$\{E^k\}$ a sequence of measured data that converges strongly to
exact data $E$ in $L^2(\Omega)$. Let be $\{(\mu^k_a, \mu^k_s)\}$ the
respective sequence of minimizers of $\J_\alpha$ with $E^\delta$
replaced by $E^k$. Then, $\{(\mu^k_a, \mu^k_s)\}$ has a convergent
subsequence and the limit of every convergent subsequence is a
minimizer of $\J_\alpha$ in $D(F) \cap [H^1(\Omega)]^2$.

\noindent \textbf{[Convergence]} Let be $\{E^k\}$ a sequence of
measured data satisfying (\ref{eq:noise}), with $\delta$ replaced by
$\delta_k$. If $\delta_k \to 0$ and the regularization parameter is
chosen such that $\delta_k^p/ \alpha(\delta_k) \to 0$, then any
sequence of minimizers of the Tikhonov functional $\J_\alpha$ with
$E^\delta$ replaced by $E^k$ has a convergent subsequence. Moreover,
the limit of every convergent subsequence is compatible with the
data and has a minimum distance to the \textit{a priori} guess
$(\mu_{a,0}, \mu_{s,0})$. This limit is called an $(\mu_{a,0},
\mu_{s,0})$-minimum-norm solution and denoted by $(\mu^\dag_a,
\mu^\dag_s)$.
\end{theorem}

It is possible to obtain quantitative convergence results if some
\textit{a priory} smoothness of the solution is required. It is
known as source condition and read as follows: Let be $(\mu^\dag_a,
\mu^\dag_s)$ a $(\mu_{a,0}, \mu_{s,0})$-minimum-norm solution.
Assume that $F$ has a directional derivative at $(\mu^\dag_a,
\mu^\dag_s)$ and denote the adjoint of $F'$ by $F'[\mu^\dag_a,
\mu^\dag_s]^*$. Moreover, assume that there exists an element $w \in
L^2(\Omega)$ such that
\begin{eqnarray}\label{eq:source}
(\mu^\dag_a, \mu^\dag_s) - (\mu_{a, 0}, \mu_{s,0}) = F'[\mu^\dag_a,
\mu^\dag_s]^* w\,, \mbox{ and }\,  C \|w\|_{L^2(\Omega)} \leq 1\,,
\end{eqnarray}
where $C$ is a constant which depends only on the boundedness of the
coefficients and the source of (\ref{eq:RTE}). Then the classical
convergence rates result \cite[Theorem 10.4]{EngHanNeu96} holds
\begin{eqnarray*}
\|F(\mu^{\delta, \alpha}_a, \mu^{\delta, \alpha}_a) -
E^\delta\|_{L^p(\Omega)} = O(\delta^{1/p})
\end{eqnarray*}
 and
\begin{eqnarray*}
  \|(\mu^{\delta, \alpha}_a, \mu^{\delta, \alpha}_a) -
(\mu^\dag_a, \mu^\dag_s)\|_{[H^1(\Omega)]^2} = O(\sqrt{\delta})\,.
\end{eqnarray*}
However, the source condition~(equation~\ref{eq:source}) is hard to be
verified in practice.

\subsection{Piecewise constant coefficient: A level set regularization
approach}\label{subsec:pc_coeff}

PAT is particularly well-suited for imaging the non-smooth structure
of the blood vasculature. In this case, the absorption and
scattering coefficients are well approximated by piecewise constant
functions.

For easy of notation, in this article we will assume that the pair
of absorption and scattering parameters $(\mu_a,\mu_s)$ has two
distinct unknown values, i.e. $\mu_a(x) \in \{a^1,a^2\}$ and
$\mu_s(x) \in \{c^1,c^2\}$ a.e. in $\Omega \subset \mathcal{R}^n$.
Therefore, we can assume the existence of open and mensurable sets
$\A_1 \subset \subset \Omega$ and $\C_1 \subset \subset \Omega$,
with $\mathcal{H}^1(\partial \A_1) < \infty \mbox{ and }
\mathcal{H}^1(\partial \C_1) < \infty$,\footnote{Here
$\mathcal{H}^1(\mathcal{S})$ denotes the one-dimensional
Hausdorff-measure of the set $\mathcal{S}$.} s.t. $\mu_a(x) = a^1\,,
x \in \A_1$, $\mu_s(x) = c^1\,, x \in \C_1$  and $\mu_a(x) = a^2\,,
x \in \A_2:= \Omega - \A_1$, $\mu_s(x) = c^2\,, x \in \C_2: = \Omega
- \C_1$. Hence, the pair of piecewise constant absorption and
scattering coefficients can be written as
\begin{eqnarray} \label{eq:def-a-c}
( \mu_a(x),\mu_s(x))  = (a^2 + (a^1-a^2) \chi_{\A_1}(x), c^2 +
(c^1-c^2) \chi_{\C_1}(x)) \, ,
\end{eqnarray}
where $\chi_{\mathcal{S}}$ is the indicator function of the set
$\mathcal{S}$.

In order to model the space of admissible parameters (the pair of
piecewise constant function $(\mu_a(x),\mu_s(x))$), we use a
standard level set (sls) approach proposed in \cite{FSL05, CLT11,
CLT09A, CLT09, CL2010}. According to this representation strategy, a
pair of real valued functions $(\phia, \phic)  \in [H^1(\Omega)]^2$
is chosen in such way that its zero level-set $ \{x \in \Omega \, ;
\ \phia(x) = 0\}$ and $\{x \in \Omega \, ; \ \phic(x) = 0\}$ define
connected curves within $\Omega$ and that the discontinuities of the
parameters are located 'along' the zero level set of $\phia$ and
$\phic$, respectively.

The piecewise constant requirement for the pair of coefficients
$(\mu_a, \mu_s)$ is obtained by introducing the Heaviside projector
$H(t)$
%
which allows us to represent the absorption and scattering
coefficients as 
\begin{eqnarray} \label{eq:def-P}
(\mu_a(x), \mu_s(x))  &  =  \ (a^1 H(\phia) + a^2(1-H(\phia)), c^1
H(\phic) + c^2(1-H(\phic)))\\
                    &   =: \ P(\phia, \phic, \vb_{ij}) \, ,
                    \nonumber
\end{eqnarray}
where $\vb_{ij}$ represents the vector of constant values
$\vb_{i,j}:= (a^1, a^2, c^1, c^2) \in \R^4$.

Within this framework, the inverse problem in (\ref{eq:operator})
with data given by equation~(\ref{eq:noise}), can be written in the
operator equation form
\begin{eqnarray} \label{eq:inv-probl-fps}
F(P(\phia, \phic, \vb_{ij}) ) \ = E^\delta\,.
\end{eqnarray}

Notice that, if an approximate solution $(\phia, \phic, \vb_{ij})$
of (\ref{eq:inv-probl-fps}) is calculated, a corresponding
approximate solution of (\ref{eq:operator}) is obtained in a
straightforward way: $(\mu_a, \mu_s) = P(\phia, \phic, \vb_{ij})$.

We remark that the analysis of level set approach for the pair of
parameter which has many piecewise components follows essentially
from the techniques derived in this approach with the multi-level
framework approach in \cite{CLT09}. Therefore we do not go through
the details here. 

For guarantee a stable approximate solution for the operator
equation (\ref{eq:inv-probl-fps}) we introduce the energy functional
\begin{eqnarray} \label{eq:Tikhonov-functional}
{\cal F}_\alpha (\phia, \phic, \vb_{ij}) := & \| F (P(\phia, \phic,
\vb_{ij})) - E^\delta \|^2_{L^2(\Omega)} + \alpha f(\phia, \phic,
\vb_{ij}) \,,
\end{eqnarray}
where $\alpha > 0$ plays the role of a regularization parameter and
$$f(\phia, \phic, \vb_{ij}) = |H(\phia)|_\bv + |H(\phic)|_\bv + \|
\phia - \phia_{,0} \|^2_{H^1(\Omega)} + \| \phic - \phic_{,0}
\|^2_{H^1(\Omega)}$$ $ + \|\vb_{ij}\|^2_{\R^4}$ is the
regularization functional. This approach is based on TV-$H^1$
penalization. The $H^1$--terms act simultaneously as a control on
the size of the norm of the level set function and as a
regularization on the space $H^1(\Omega)$. The $\bv$-seminorm terms
are well known for penalizing the length of the Hausdorff measure of
the boundary of the sets $\{x \in \Omega \,:\, \phia(x) > 0 \}$,
$\{x \in \Omega \, : \, \phic(x) > 0 \}$ (see \cite{EG92}). Others
level set approaches have been applied to recover piecewise constant
function in the literature, e.g. \cite{DL09, CL2010, CLT11, DA06}
and references therein.

In general, variational minimization techniques involve compact
embedding arguments and the continuity of the forward operator on
the set of admissible minimizers to guarantee the existence of
minimizers. The Tikhonov functional in
(\ref{eq:Tikhonov-functional}) does not allow such characteristic,
since the Heaviside operator $H$ and consequently the operator $P$
in equation~\ref{eq:def-P} are discontinuous. Therefore, given a minimizing
sequence $(\phi_a^k, \phi_c^k,\vb_{ij}^k)$ for ${\cal F}_\alpha$ we
cannot prove existence of a (weak-*) convergent subsequence.
Consequently, we cannot guarantee the existence of a minimizer in
$[H^1(\Omega)]^2 \times \R^4$. To overcome this difficulty we follow
\cite{FSL05, CLT09A, CLT09} and introduce the concept of generalized
minimizers in order to guarantee the existence of minimizers of the
Tikhonov functional (\ref{eq:Tikhonov-functional}).

First we introduce a smooth approximation of the Heaviside
projection given by
$$
H_\ve(t) := \left\{
  \begin{array}{rl}
    1 + t/\ve  & \mbox{ for \ } t \in \left[-\ve,0\right] \\
    H(t)       & \mbox{ for \ } t \in \mathcal{R} / \left[-\ve,0\right] \\
\end{array} \right.
$$
and the corresponding operator
\begin{eqnarray} \label{eq:def-Pve}
\quad P_{\ve}(\phia, \phic,\vb_{ij}) \ := \ (a^1 H_\ve(\phia) + a^2
(1 - H_\ve(\phia)),  c^1 H_\ve(\phic) + c^2 (1 - H_\ve(\phic))) \,,
\end{eqnarray}
for each $\ve > 0$. Then:

\begin{definition} \label{def:vector}
Let be the operators $H$, $P$, $H_\ve$ and $P_{\ve}$ defined as
above.\\
\medskip
 \noindent i) A {\bf vector} $(z_1, z_2,\phia, \phic,\vb_{ij}) \in
[L^\infty(\Omega)]^2 \times [H^1(\Omega)]^2 \times \mathcal{R}^2$ is
called {\bf admissible} when there exists sequences $\{ \phi_a^k \}$
and $\{ \phi_s^k \}$ of $H^1(\Omega)$-functions satisfying
$$\lim\limits_{k\to\infty} \| \phi_a^k - \phia \|_{L^2(\Omega)} =
0\,,\quad \lim\limits_{k\to\infty} \| \phi_s^k - \phic
\|_{L^2(\Omega)} =  0$$ and there exists a sequence $\{ \ve_k \} \in
\mathcal R^+$ converging to zero such that
$$\lim\limits_{k\to\infty} \| H_{\ve_k}(\phi_a^k)-z_1
\|_{L^1(\Omega)} = 0\, \mbox{ and } \lim\limits_{k\to\infty} \|
H_{\ve_k}(\phi_s^k)-z_2 \|_{L^1(\Omega)} = 0\,.$$
\noindent b) A {\bf generalized minimizer} of the Tikhonov
functional $\mathcal{F}_\alpha$ in (\ref{eq:Tikhonov-functional}) is
considered to be any admissible vector $(z_1, z_2,\phia,
\phic,\vb_{ij})$ minimizing
\begin{eqnarray} \label{eq:gzphi}
\qquad {\cal{G}}_\alpha(z_1, z_1,\phia, \phic,\vb_{ij}) :=  \|
F(q(z_1,z_2,\vb_{ij})) - E^\delta \|^2_{L^2(\Omega)} + \alpha
     R(z_1,z_2,\phia, \phic,\vb_{ij})
\end{eqnarray}
over the set of admissible vectors, where $$q: [L_\infty(\Omega)]^2
\times \mathcal R^2 \ni (z_1,z_2,\vb_{ij}) \mapsto (a^1 z_1 +
a^2(1-z_2), c^1 z_2 + c^2 (1-z_2)) \in [L_{\infty}(\Omega)]^2\,,$$
and the functional $R$ is defined by
\begin{eqnarray} \label{def:R}
R(z_1, z_2,\phia, \phic,\vb_{ij}) \ := \ \rho(z_1,z_2,\phia, \phic)
+ \|\vb_{ij}\|^2_{\mathcal R^2} \, ,
\end{eqnarray}
with
\begin{eqnarray*}
\rho(z_1,z_2,\phia, \phic) & & := \inf \left\{ \liminf_{k\to\infty}
\left(|H_{\ve_k} (\phi_a^k)|_\bv + |H_{\ve_k} (\phi_s^k)|_\bv
\right. \right. \\ & &  \qquad \qquad \qquad\qquad \qquad \qquad
\left. \left.  + \|( \phi^k_a, \phi^k_s) - (\phia_{,0},
\phic_{,0})\|^2_{[H^1(\Omega)]^2} \right) \right\}\,.
\end{eqnarray*}
Here the infimum is taken over all sequences $\{\ve_k\}$ and
$\{\phi_a^k, \phi_s^k\}$ characterizing the vector $(z_1,z_2,\phia,
\phic,\vb_{ij})$ as an admissible vector.
\end{definition}

Given the continuity of the operator $F$ in $L^1(\Omega)$ (see
Theorem~\ref{theo:continuity}), we can follow the proofs in
\cite{CLT09A, CLT09} to guarantee the existence, stability and
convergence of approximated solutions for the inverse problem~
(\ref{eq:inv-probl-fps}). For the sake of completeness, we collect
the results without proving.

\begin{theorem} Let`be $p,q,r$ satisfying the assumption of Corollary~\ref{coro:continuity}. Then the following assertions hold true.

\noindent \textbf{Existence:} The functional $\Ga$ in
(\ref{eq:gzphi}) attains minimizers on the set of admissible
vectors.

\noindent \textbf{Convergence for exact data:} Assume that we have
exact data, i.e. $E^\delta=E$. For every $\alpha > 0$ denote by
$(z^1_\alpha, z^2_\alpha, \phia_\alpha,
\phic_\alpha,\vb_{{ij},\alpha})$ a minimizer of $\Ga$ on the set of
admissible vectors. Then, for every sequence of positive numbers
$\{\alpha_k\}$ converging to zero there exists a subsequence,
denoted again by $\{\alpha_k\}$, such that $(z^1_{\alpha_k},
z^2_{\alpha_k}, \phia_{\alpha_k},
\phic_{\alpha_k},\vb_{{ij},{\alpha_k}})$ is strongly convergent in
$[L^1(\Omega)]^2 \times [L^2(\Omega)]^2 \times \mathcal{R}^2$.
Moreover, the limit is a solution of (\ref{eq:inv-probl-fps}).

\noindent {\bf Convergence for noise data: } Let $\alpha =
\alpha(\delta)$ be a function satisfying $\lim_{\delta \to 0}$
$\alpha(\delta) = 0$ and $\lim_{\delta \to 0} \delta^2
\alpha(\delta)^{-1} = 0$. Moreover, let $\{ \delta_k \}$ be a
sequence of positive numbers converging to zero and $\{ E^{\delta_k}
\} \in L^2(\Omega)$ be corresponding noise data satisfying
(\ref{eq:noise}). Then, there exists a subsequence, denoted again by
$\{ \delta_k \}$, and a sequence $\{ \alpha_k := \alpha(\delta_k)
\}$ such that $(z^1_{\alpha_k}, z^2_{\alpha_k},\phia_{\alpha_k},
\phic_{\alpha_k}, \vb_{{ij},{\alpha_k}})$ converges in
$[L^1(\Omega)]^2 \times [L^2(\Omega)]^2 \times \mathcal{R}^2$ to
solution of (\ref{eq:inv-probl-fps}).
\end{theorem}

\begin{remark} \label{rem:ad-top}
The set of admissible vector is to be considered as a topological
space, namely a subset of $[L^\infty(\Omega)]^2 \times [H^1(\Omega)
]^2\times \mathcal{R}^4$ endowed with the topology of
$[L^1(\Omega)]^2\times [ L^2(\Omega)]^2 \times \mathcal{R}^4$. In
order to guarantee the existence of generalized minimizers of
$\mathcal{F}_\alpha$ one interesting properties is the  closedness
of this extended parameter space. It is analyzed in \cite{FSL05,
CLT09}.

We also remark that the definition of admissible vector (see
Definition~\ref{def:vector}) is constructed in a non-standard
manner. However, such definition implies in the closedness of the
graph of the Tikhonov functional defined in
(\ref{eq:Tikhonov-functional}) and hence the existence of a
generalized minimizer of the Tikhonov functional
$\mathcal{F}_\alpha$ in (\ref{eq:Tikhonov-functional}).

It is worth noticing that, for each $\ve> 0$ the
$L^\infty$--functions $H_{\eps} (\phi_a)$ and $H_{\eps} (\phi_s)$ in
Definition~\ref{def:vector} are elements of $D(F)$. Moreover, from
the smoothness of the level set functions, they are also in
$H^1(\Omega)$. Hence, the Fr\'echet derivative of the forward
operator $F$ in Theorem~\ref{theo:dir-diff} holds.
\end{remark}

\subsubsection{Numerical realization of the Tikhonov approach}
\label{sec:tow-num-sol} 

We remark that the Tikhonov functional $\Ga$ defined in the previous
section is not suitable for computing numerical approximations to
the solution of (\ref{eq:inv-probl-fps}). This becomes obvious when
one observes the definition of the penalization term $\rho$ in
Definition~\ref{def:vector}.

In this section we introduce the functional $\Ger$, which can be
used for the purpose of numerical implementations. This functional
is defined in such a way that it's minimizers are 'close' to the
generalized minimizers of $\Ga$ in a sense that will be made clear
later (see Proposition~\ref{th:just}). For each $\eps > 0$ we define
the functional
\begin{eqnarray} \label{eq:m-reg}
\Ger(\phia, \phic, \vb_{ij}) \ := \| F(P_\eps(\phia,\phic,
\vb_{ij})) - E^\delta \|^2_{L^2(\Omega)} +
    \alpha R_\ve(\phia, \phic, \vb_{ij})\; ,
\end{eqnarray}
where
\begin{eqnarray*}
R_\ve(\phia, \phic, \vb_{ij}) : =  & & \left( |H_\eps(\phia)|_{\bv}
+ |H_{\eps} (\phic)|_{\bv} \right.\\
 & & \qquad \qquad \qquad  \left. + \| (\phia, \phic) - (\phia_{,0}, \phic_{,0}) \|^2_{[H^1(\Omega)]^2} + \|\vb_{ij}\|^2_{R^4}\right)\,
\end{eqnarray*}
and
 $P_\eps(\phia,\phic, \vb_{ij})$ is the operator defined in (\ref{eq:def-Pve}). This
functional is well-posed as the following lemma shows:

\begin{lemma}\cite[Lemma~10]{CLT09}\label{lemma:13}
Given $\alpha > 0$, $\eps > 0$ and $(\phia_{,0}, \phic_{,0}) \in
 [H^1(\Omega)]^2$, then the functional $\Ger$ in (\ref{eq:m-reg}) attains
a minimizer on $[H^1(\Omega)]^2 \times \mathcal{R}^4$.
\end{lemma}

The next result guarantees that, for $\eps \to 0$, the minimizers of
$\Ger$ approximate a generalized minimizer of $\Ga$.

\begin{proposition}\cite[Theorem~11]{CLT09} \label{th:just}
Let $\alpha > 0$ be given. For each $\eps > 0$ denote by
$(\phia_{\eps,\alpha},\phic_{\eps,\alpha}, \vb_{{ij},{\eps,\alpha}}
)$ a minimizer of $\Ger$. There exists a sequence of positive
numbers $\{ \eps_k \}$ converging to zero such that $$(
H_{\eps_k}(\phia_{{\eps_k},\alpha}),
H_{\eps_k}(\phic_{{\eps_k},\alpha}), \phia_{{\eps_k},\alpha},
\phic_{{\eps_k},\alpha}, \vb_{{ij},{\eps_k,\alpha}})$$ converges
strongly in $[L^1(\Omega)]^2 \times [\lzo]^2 \times \mathcal{R}^4$
and the limit is a generalized minimizer of $\Ga$ in the set of
admissible vectors.
\end{proposition}

Proposition~\ref{th:just} justifies the use functionals $\Ger$ in
order to obtain numerical approximations to the generalized
minimizers of $\Ga$. It is worth noticing that, differently from
$\Ga$, the minimizers of $\Ger$ can be actually computed. 
%
%
%
%
\section{A note on the numerical realization}\label{sec:numerical}
In order to develop a computational scheme to minimize the proposed
Tikhonov functionals
(\ref{eq:Tikho-func-smooth})-(\ref{eq:Tikhonov-functional}), one way
is looking for first order optimality condition for which gradient
or Newton-type algorithms can be implemented, see \cite{STCA13} and
references therein. Here we concentrate our attention to the inner
product structure of $L^2(\Omega)$. For the proposed $L^p$
approaches with $p \in [1, 2[$ some sub-gradient type algorithm may
be used, see \cite{Neubauer} and references therein.

We first provide a formal derivative for the least square term
$J(\mu_a, \mu_s): = \frac{1}{2}\|E^\delta - F(\mu_a,
\mu_s)\|^2_{L^2(\Omega)}$ in the functionals
(\ref{eq:Tikho-func-smooth})-(\ref{eq:Tikhonov-functional}).

Using the same notation of the Theorem~\ref{theo:dir-diff}, the
derivative of the least square term of $J$ at $(\mu_a, \mu_s) \in
D(F) \cap H^1(\Omega)$ in the direction $(\t \mu_a, \t \mu_s) \in
H^1(\Omega)$ can be written as
\begin{eqnarray}\label{eq:derivative_J1}
D\,J[\t \mu_a, \t \mu_s] = - \langle E^\delta - F(\mu_a,
\mu_s), F'(\mu_a, \mu_s)[\t \mu_a, \t \mu_s] \rangle_{L^2(\Omega)} 
\end{eqnarray}

By substituting (\ref{eq:dir-derivative}) in the first therm in the
right hand side of (\ref{eq:derivative_J1}), we have
\begin{eqnarray}\label{eq:derivative_J2}
\qquad D\,J[\t \mu_a, \t \mu_s] &  = & - \langle U(\mu_a, \mu_s)(
E^\delta - F(\mu_a, \mu_s)), \t \mu_a \rangle_{L^2(\Omega)} \\
&  & - \big\langle  \mu_a (E^\delta - F(\mu_a, \mu_s)), \int_S
u'(\mu_a, \mu_s; \s)[\t \mu_a, \t
 \mu_s] d\s \big\rangle_{L^2(\Omega)}\,. \nonumber
\end{eqnarray}
Let $v$ be the solution of the adjoint problem
\begin{eqnarray}\label{eq:adjoint-RTE}
T^* v = \mu_a (E^\delta  - F(\mu_a, \mu_s))\,,
\end{eqnarray}
with absorbing boundary conditions.

The substitution of (\ref{eq:adjoint-RTE}) into
equation~(\ref{eq:derivative_J2}) produces
\begin{eqnarray}\label{eq:derivative_J3}
D\, J[\t \mu_a, \t \mu_s]  & = & - \langle U(\mu_a, \mu_s)(
E^\delta - F(\mu_a, \mu_s)), \t \mu_a \rangle_{L^2(\Omega)} \\
  & & - \big\langle T^* v , \int_S u'(\mu_a, \mu_s; \s')[\t \mu_a, \t
 \mu_s] d\s' \big\rangle_{L^2(\Omega)} \,.\nonumber
\end{eqnarray}

Since $T^*v$ does not depend on the direction $\s'$, the
equation~(\ref{eq:derivative_J3}) can be written equivalently as
\begin{eqnarray}\label{eq:derivative_J31}
D\, J[\t \mu_a, \t \mu_s]  = & - \langle U(\mu_a, \mu_s)(
E^\delta - F(\mu_a, \mu_s)), \t \mu_a \rangle_{L^2(\Omega)} \\
 & - \langle T^* v , u'(\mu_a, \mu_s; \s')[\t \mu_a, \t
 \mu_s] \rangle_{L^2(D)} \,.\nonumber
\end{eqnarray}

Since $u'(\mu_a, \mu_s) = 0$ on $\partial \Omega$ (see
(\ref{eq:u'})), by using the divergence theorem, we can state that
\begin{eqnarray*}
0 & = & \int_{\partial \Omega} (\s \cdot \eta) v(x, \s) u'(x, \s)[\t
\mu_a, \t \mu_s] dx = \int_\Omega (\s \cdot \nabla) (v( x ,\s) u'(x, \s)[\t \mu_a, \t \mu_s]) dx \nonumber \\
& = & \int_\Omega v(x, \s) (\s \cdot \nabla) u' (x, \s)[\t \mu_a, \t
\mu_s] +  u' (x, \s)[\t \mu_a, \t \mu_s]  (\s \cdot \nabla) v(x, \s)
dx \\
& = & \langle v, Tu' \rangle_{L^2(\D)} - \langle T^* v, u'
\rangle_{L^2(\D)}\,. \nonumber
\end{eqnarray*}
From (\ref{eq:u'}) we have that
\begin{eqnarray}\label{eq:derivative_J6}
\langle v, Tu' \rangle_{L^2(\D)} & = \int_D v(x, \s) \left((\t \mu
_a + \t \mu_s) u(x, \s) - \t \mu_s \int_S \Theta(\s,
\s') u(x, \s') d\s' \right) dx d\s  \nonumber \\
& = \big\langle u v, \t \mu_a + \t \mu_s\rangle_{L^2(\D)} - \langle
u \int_S \Theta(\s, \s') v(\cdot, \s') d\s' , \t \mu_s
\big\rangle_{L^2(\D)}
\end{eqnarray}
From (\ref{eq:derivative_J2}) - (\ref{eq:derivative_J6}) we obtain
that
\begin{eqnarray}\label{eq:derivative_J7}
D\, J[\t \mu_a, \t \mu_s] & = & - \langle U(E^\delta  - F(\mu_a,
\mu_s), \t \mu_a \rangle_{L^2(\Omega)} +  \langle u v, \t \mu_a
\rangle_{L^2(\D)} \\
 & & + \big\langle u v, \t \mu_s\rangle_{L^2(\D)} - \langle u \int_S
\Theta(\s, \s') v(\cdot, \s') d\s , \t \mu_s
\big\rangle_{L^2(\D)}\,. \nonumber
\end{eqnarray}

A quick calculation shows that the derivative of the regularization
term satisfies
\begin{eqnarray}\label{eq:derivative_R}
\langle (I - \Delta) (\mu_a - \mu_{a,0}), \t \mu_a
\rangle_{L^2(\Omega)}   + \langle (I - \Delta) (\mu_s - \mu_{s,0}),
\t \mu_s \rangle_{L^2(\Omega)} \\
 =    - \int_{\partial \Omega} \frac{\partial}{\partial \eta}(\mu_a
- \mu_{a,0}) \t \mu_a dx - \int_{\partial \Omega}
\frac{\partial}{\partial \eta}(\mu_s - \mu_{s,0}) \t \mu_s dx \,.
\nonumber
\end{eqnarray}

Hence, the gradient of $\J_\alpha$ with respect to the absorption
and scattering coefficients can be formally written as
\begin{eqnarray} \label{eq:derivative_Ja}
\frac{\partial \J_\alpha}{\partial \mu_a}(\mu_a, \mu_s) & = &  -
U(\mu_a,
\mu_s)(E^\delta - F(\mu_a, \mu_s)) \\
& + &  \int_S u(\cdot, \s) v(\cdot,\s)d\s + 2 \alpha (I - \Delta) (\mu_a - \mu_{a,0})\,, \nonumber\\
\label{eq:derivative_Js} \frac{\partial \J_\alpha}{\mu_s}(\mu_a,
\partial \mu_s) & = & \int_S u(\cdot, \s) v(\cdot, \s) d\s \\ & - & \int_S
\int_S u(\cdot, \s) \Theta(\s, \s') v( \cdot, \s) d\s d\s' + 2
\alpha (I - \Delta) (\mu_s - \mu_{s,0})\,, \nonumber
\end{eqnarray}
subject to the homogeneous Neumann boundary conditions
\begin{eqnarray}\label{eq:Neumann}
\frac{\partial}{\partial \eta}(\mu_a - \mu_{a,0})  = 0\,,\qquad
\frac{\partial}{\partial \eta}(\mu_s - \mu_{s,0})  = 0\,,
\end{eqnarray}
respectively.

In \cite{STCA13} limited-memory BFGS was used to solve
(\ref{eq:derivative_Ja})-(\ref{eq:derivative_Js}) while a finite
element model of the (RTE) equation~(\ref{eq:RTE}) was  used to
determine the optical absorption and scattering coefficients.


\subsection{Optimality conditions for the Tikhonov functional $\Ger$}

For the numerical implementation of the level set approach is
necessary to derive the first order optimality conditions for a
minimizer of the functionals $\Ger$. With this finality, we consider
$\Ger$ in (\ref{eq:m-reg}) and we look for the G\^ateaux directional
derivatives with respect to $\phia$, $\phic$. For easiness of
presentation, we assume here that the constant values $\vb_{ij}$ are
known. An algorithm for unknown constant values was implemented in
\cite{CLT09A}.

Given the composition of $P_\ve$ with the forward operator $F$ in
the level set approach and the self-adjointness\footnote{Notice that
$H'_\eps(t) =  1/\eps \mbox{ if }  t \in (-\eps, 0)\,,  0 \mbox{
else}\,.$} of $H'_\eps(\varphi)$, the optimality conditions for a
minimizer of the functional $\Ger$ can be written in the form of the
system of equations
\begin{eqnarray}
\frac{\partial }{\partial \phia} \Ger(\phia, \phic) =
L_{\eps,\alpha}^1(\phia,\phic) + \alpha (I - \Delta)(\phia -
\phia_{,0}) \,,\\
\frac{\partial }{\partial \phic} \Ger(\phia, \phic) =
L_{\eps,\alpha}^2(\phia,\phic) + \alpha (I - \Delta)(\phic -
\phic_{,0}) \,,
\end{eqnarray}
%
with the homogeneous Neumann boundary condition
\begin{eqnarray}
\frac{\partial} {\partial \eta} (\phia - \phia_{,0}) = 0 \,, \qquad
\frac{\partial} {\partial \eta}(\phic - \phic_{,0}) = 0
\end{eqnarray}
where
%
\begin{eqnarray*}
L_{\eps,\alpha}^1(\phia,\phic) & = &  (a_1 - a_2) \,
H'_\eps(\phia)\, \frac{\partial J}{\partial \phia}(P(\phia, \phic,
\vb_{ij})) + \alpha \, \left[ H'_\eps(\phia) \, \nabla \!\cdot\!
     \left( \frac{\nabla H_\eps(\phia)}{|\nabla H_\eps (\phia)|} \right) \right] \\
L_{\eps,\alpha}^2(\phia,\phic) & = &  (c_1 - c_2) \,
H'_\eps(\phic)\, \frac{\partial J}{\partial \phic}(P(\phia, \phic,
\vb_{ij})) + \alpha \, \left[ H'_\eps(\phic) \, \nabla \!\cdot\!
\left( \frac{\nabla H_\eps(\phic)}{|\nabla H_\eps (\phic)|} \right)
\right]
\end{eqnarray*}
%
and
$\frac{\partial J}{\partial \phia}(P(\phia, \phic, \vb_{ij}))$ and
$\frac{\partial J}{\partial \phic}(P(\phia, \phic, \vb_{ij}))$ are
obtained analogous to (\ref{eq:derivative_J7}) for absorption and
scattering coefficients parameterized by $P_\ve(\phia, \phic,
\vb_{ij})$.

\subsection{The real case: Multiple illumination positions}

As it is well known for the diffusive approximation for (PAT),
non-uniqueness may be encountered when both absorption and
scattering coefficients are recovered without include additional
information into the problem \cite{BU10, BR11}. One additional
information generally used to avoid the non-uniqueness is using a
multiple - illumination approach, whereby a set of images are
obtained using sources placed at different positions around the
image domain. It was also reported in \cite{BR11} that the
multiple-illumination approach improves the ill-posedness of the
(QPAT) inverse problem.

A few theoretical modification of the presented approach implies the
stated results. Indeed, if $N_m$ is the number of source positions,
then one alternative is looking for the Tikhonov-type functionals
(\ref{eq:Tikho-func-smooth}) - (\ref{eq:Tikhonov-functional}) with
the misfit replaced by
\begin{eqnarray}
\qquad \sum_{m=1}^{N_m} \frac{1}{p}\|E^{\delta}_m - F_m(\mu_a,
\mu_s)\|^p_{L^p(\Omega)} \,\, \mbox{ and } \, \, \sum_{m=1}^{N_m}
\frac{1}{2}\|E^{\delta}_m - F_m(P(\phia, \phic,
\vb_{ij})\|^2_{L^2(\Omega)}\,,
\end{eqnarray}
respectively. Another alternative is writing the problem as a system
of nonlinear operator equations
\begin{eqnarray}
F_m(\mu_a, \mu_s) = E_m^\delta\,,\quad m=1\, \cdots\, N_m\,,
\end{eqnarray}
and then uses a Kaczmarz-type strategy \cite{BDCL2010} for
regularize the problem.

\section{Conclusions and future directions}\label{sec:conclusion}

Existence and stability of approximated solution have been shown to
determine the absorption and scattering coefficients in the (RTE)
model for (QPAT), using Tikhonov-type regularization approaches.
Sufficient conditions to obtain the regularization properties of the
approximated solution has been shown by proving properties of the
forward problem as continuity, compactness and Fr\'echet derivative.
The results concern with different topologies which includes the
physical and numerical issues. Although we do not present any
implementation, our results imply in the theoretical guarantee of
regularization properties of the approach presented in
\cite{STCA13}. Using \textit{a priori} information of the regularity
of the coefficient, we propose different Tikhonov-type
regularization. We are confident that the level set approach for
piecewise continuous coefficient can improve significatively the
results in \cite{STCA13}.

Since PAT is a new imaging modality, many theoretical and
computational issues are still open, e.g. \cite{KK2008, STCA13,
TKPVSAK10, TZC2010, SU2009, BU10, BRUZ11, BR11}. The numerical
implementation of the regularization approaches which has been
proposed in this contribution will be the subject of future work.
Given the hyperbolically nature of the (RTE) equation may imply in
numerical instability of standard Galerking discretizations. In
particular, it will bring the discussion on the numerical issue in a
real PAT image \cite{STCA13}.

As far as the authors known, iterative regularization
\cite{KaltNeuScher08, EngHanNeu96} was not attempted in the PAT
context. It shall be considered in future works.

\section{Acknowledgments}
A.D.C. acknowledges support from CNPq - Science Without Border grant
200815/2012-1 and from ARD-FAPERGS grant 0839 12-3. Part of this
work was conducted during the post-doc of A.D.C. at the Department
of Computer Science at UBC. F.T.C acknowledges support from FAPERGS
grant PqG 0839 12-3.
%
%

\section{Bibliography}
\bibliographystyle{elsarticle-harv}
\bibliography{pat}

\end{document}